\documentclass[reqno,11pt,final]{amsart}
\usepackage{style}
\title[Kolmogorov equations associated to filtering equation: viscosity solutions]{Kolmogorov equations on the space of probability measures associated to the nonlinear filtering equation: the viscosity approach}

\author{Mattia Martini}
\address{Mattia Martini: Dipartimento di Matematica "Federigo Enriques", Università degli Studi di Milano, Via Saldini 50,  20133 Milano (Italy)}
\email{mattia.martini@unimi.it}

\begin{document}
\begin{abstract}
We study the backward Kolmogorov equation on the space of probability measures associated to the Kushner-Stratonovich equation of nonlinear filtering. We prove existence and uniqueness in the viscosity sense and, in particular, we provide a comparison theorem. 

In the context of stochastic filtering it is natural to consider measure-valued processes that satisfy stochastic differential equations. In the literature, a classical way to address this problem is by assuming that these measure-valued processes admit a density. Our approach is different and we work directly with measures. Thus, the backward Kolmogorov equation we study is a second-order partial differential equation of parabolic type on the space of probability measures with compact support. In the literature only few results are available on viscosity solutions for this kind of problems and in particular the uniqueness is a very challenging issue. Here we find a viscosity solution and then we prove that it is unique via comparison theorem.
\end{abstract}

\maketitle
\section{Introduction}
The main goal of this paper is to discuss existence and uniqueness for viscosity solutions to a class of backward Kolmogorov equations on spaces of probability measures, arising in the context of stochastic filtering.\\

The study of partial differential equations on spaces of probability measures is a topic of increasing research interest. For instance, in \cite{ambrosiogangbo,gangbotudorascu,bayraktarcossopham,phamwei18,cardaliaguetcirantporretta} and references therein one can find many results both for the first and second-order case. A main issue is how to define the derivatives on this spaces and many approaches are available in the literature (see \cite{ambrosiogiglisavare} for a more geometric approach related to optimal transport or \cite{lions, carmonadelarue1} for a more functional analytic and probabilistic approach connected with the theory of mean field games).

Here we continue the study started in \cite{martini}. In particular, we consider a stochastic differential equation for a probability measure-valued process arising in the context of nonlinear filtering, and then we study the associated backward Kolmogorov equation, that is a second-order linear partial differential equation of parabolic type on the space of probability measures. We prove the existence and uniqueness of a viscosity solution to this equation.\\ 

The stochastic filtering problem is about finding a way to approximate a signal (described by a stochastic process) that can not be directly observed. In the last decades a lot of research interest has focused in this direction, see for instance \cite{baincrisan, xiong} and references therein for an exhaustive introduction to the topic. A key object for the theory is the so-called filter, a probability measure-valued stochastic process defined, at any time, as the conditional law of the signal process given the available information. One can prove that this process is Markovian and satisfies a stochastic differential equation (in a way we will point out later), called Kushner-Stratonovich equation. Then, it is natural to study the associated backward Kolmogorov equation.

A classical way to approach this problem (see for instance \cite{rozovsky}), is to show that the filter admits a density which belongs to a suitable Hilbert space, and then exploit the tools of stochastic calculus available in this case. However, this requires the introduction of unnecessary hypotheses. In this paper we avoid these conditions and rather follow the approach of \cite{baincrisan,heunislucic, szpirglas}, where the filtering process is studied as genuine measure-valued process.

Let us fix a final time $T>0$, a $d$-dimensional Brownian motion $I = \{I_t,t\in[0,T]\}$ and let us denote with $\pi=\{\pi_t,t\in[0,T]\}$ the filter, which takes values in $\prob(\R^d)$, the space of probability measures over $\R^d$. The Kushner-Stratonovich equation reads as follow: for every test function $\varphi$, it holds
\begin{equation}\label{eqn: ksintro}
	\de \scalprod{\pi_t}{\varphi} = \scalprod{\pi_t}{A\varphi}\de t + (\scalprod{\pi_t}{h\varphi + B\varphi}^\top-\scalprod{\pi_t}{\varphi}\scalprod{\pi_t}{h}^\top)\de I_t.
\end{equation}
Here $\scalprod{\mu}{\varphi}:=\int\varphi(x)\mu(\de x)$, $A$ and $B$ are respectively a second and first-order differential operators whilst $h\colon\R^d\to\R^d$. The complete and rigorous framework is provided in Section \ref{subsec: filt}. Under hypotheses that guarantee the existence and uniqueness of a Markovian solution to \eqref{eqn: ksintro}, it has been proved in \cite{martini} that the Kolmogorov operator associated to the Kushner-Stratonovich equation has the following form:
\begin{equation*}
	(\mathcal{L}u)(\mu) = \scalprod{\mu}{A\lf u(\mu)} + \frac{1}{2}\scalprod{\mu\otimes\mu}{(h+B-\scalprod{\mu}{h})^\top(h+B-\scalprod{\mu}{h})\lf^2u(\mu)},
\end{equation*}
where $\lf u$ and $\lf^2 u$ are suitably defined derivatives of real-valued functions over $\prob(\R^d)$ (see Section \ref{ssec: deriv} for the precise definitions). Thus, given a final condition $\Phi\colon\prob(\R^d)\to\R$, the backward Kolmogorov equation reads as
\begin{equation}\label{eqn: kolmintr}
	\begin{cases}
		-\partial_t u(\mu,t) - \mathcal{L} u (\mu,t) = 0,\quad &(\mu,t)\in\prob(\R^d)\times[0,T),\\
		u(\mu,T) = \Phi(\mu),\quad &\mu\in\prob(\R^d).
	\end{cases}
\end{equation}
In \cite{martini}, the author proved that if $\Phi$ is twice continuously differentiable, then there exists a unique classical solution to \eqref{eqn: kolmintr}.\\

The aim of this paper is to study the case of a less regular final condition, namely just continuous and bounded. In this case, the notion of classical solution turns out to be too strong, thus a weaker notion is necessary. In particular, the notion of viscosity solution seems to be the natural one for this kind of problems. 

The theory of viscosity solutions is well-established in the finite dimensional case (see for instance \cite{crandallishiilions}) and in the infinite dimensional Hilbert case (see for instance \cite{swiech94}). Recently some developments have been obtained for more complicate spaces, where it is not easy to provide a notion of derivative. Two examples are the space of continuous paths (see for instance \cite{cossorusso} and references therein) and the space of probability measures, to which the present work refers. In particular, in \cite{bayraktarcossopham,phamwei17,phamwei18,cossogozzikharroubiphamrosestolato1,bandinicossofuhrmanpham2} the existence of viscosity solutions to second-order equations on the space of probability measures is addressed. On the other hand, uniqueness results are very hard to achieve, and only few papers are available for second-order equations (\cite{burzoniignazioreppensoner, wuzhang, cossogozzikharroubiphamrosestolato2}).

Our work gives a contribution in this direction. In Proposition \ref{prop: existence} we show that there exists a viscosity solution to \eqref{eqn: kolmintr}. Then, we are also able to prove a uniqueness result via comparison principle (Theorem \ref{thm: comparison} and Corollary \ref{cor: uniq}). The technique is based on the fact that we have a candidate solution (given by a probabilistic representation formula of the usual type) that can be used to obtain partial comparison principles. An assumption we need in order to achieve this last result, is to restrict ourselves to the compact subspace $\prob(K)\subset\prob(\R^d)$, the space of probabilities supported on a compact set $K\subset \R^d$. This is for instance the case when the filtering equation is associated to a signal with trajectories confined on a compact subset or on a compact manifold like a torus. However, a remarkable feature of this approach is to avoid the lifting of the problem on a Hilbert space, which is a procedure often used to solve this type of second-order problems. In that case, for viscosity solutions it is not yet clear what is the relation between the lifted problem and the original one.\\ 

The rest of the paper is organized as follows. After a preliminary part on the notations, in Section \ref{sec: prelim} we clarify the notion of derivative on the space of probability measure we will use and then we introduce briefly the stochastic filtering framework of our interest. In Section \ref{sec: viscosity} we introduce the backward Kolmogorov equation we want to study, in Section \ref{ssec: exist} we prove that there exists a viscosity solution and in Section \ref{ssec: uniq} we show the comparison principle and the uniqueness. Finally, in Appendix \ref{app: analytical} we collect some analytical results needed in the paper.
\section{Notations and preliminaries}\label{sec: prelim}
Let $E\subseteq\R^d$. We denote with $\mathrm{C}(E)$ the space of continuous real-valued functions over $E$, with $\mathrm{C_b}(E)$ the continuous and bounded real-valued functions over $E$ and with $\mathrm{C_c}(E)$ the continuous functions over $E$ with compact support. Similarly, for $k\in\N$, we denote with $\mathrm{C}^k(E)$ (resp. $\mathrm{C}^k_{\mathrm{b}}(E)$ and $\mathrm{C}^k_{\mathrm{c}}(E)$) the space of $k$-time continuously differentiable (resp. continuously differentiable with bounded derivatives up to order $k$-th and continuously differentiable with compact support) functions over $E$. If $E$ is compact, we say that a function is continuously differentiable over $E$ if it has bounded and uniformly continuous derivative over $\Int{E}$.

Let $u$ be a real-valued (or $\R^d$-valued) function over a suitable space $X$. We denote with $\norm{u}_\infty:=\sup_{x\in X} \lvert u(x)\vert$ its supremum norm.\\  

We denote with $\prob(E)$ the space of probability measures over $E$ and with $\scalprod{\mu}{\varphi}$ or $\mu(\varphi)$ the integral of $\varphi\in\mathrm{C_b}(E)$ with respect to $\mu\in\prob(E)$. We say that a sequence $\{\mu_n\}_{n\in\N}\subset\prob(E)$ converges weakly to $\mu$ if $\scalprod{\mu_n}{\varphi}\to\scalprod{\mu}{\varphi}$ as $n\to+\infty$, for every $\varphi\in\mathrm{C_b}(E)$. If we take $E$ compact, we have that the sets $\prob(E)$ coincides with the set of probability measures with finite moment of any order. Moreover, one can prove that if $E$ is compact, then $\prob(E)$ is also compact in the weak topology (see for instance \cite{villani}).

Let us consider a family $\{f_k\}_{k\in\N}$ of functions dense in $\mathrm{C^\infty_c}(E)$, with $\sup_{x\in E}\lvert f_k(x)\rvert=:\norm{f_k}_\infty\leq 1$, and let us assume that the function identically equal to $1$ belongs to this family. Let us introduce a sequence $\{q_k\}_{k\in\N}\subset[1,\+\infty)$ that we will make it explicit in Appendix \ref{app: analytical}. We define $\sfd_2\colon\prob(E)\times\prob(E)\to[0,+\infty)$ as
	\begin{equation}\label{eqn: distwang}
		\sfd_2(\mu,\nu) :=\left( \sum_{k=1}^{+\infty}\frac{1}{2^kq_k}\scalprod{\mu-\nu}{f_k}^2\right)^{\frac{1}{2}}.
	\end{equation}
It is easy to show that $\sfd_2$ is a distance over $\prob(E)$ and, if we take $E$ compact, one can also prove that it is complete (see \cite[Theorem 2.4.1]{borkar}). Moreover, $\sfd_2$ metricize the weak convergence (\cite[Theorem 2.1.1]{borkar}).\\

\noindent For the rest of the paper, we will denote with $K$ a compact subset of $\R^d$.

\subsection{Notions of derivative on $\prob(K)$}\label{ssec: deriv}
Let $K\subset\R^d$ be compact. We need to introduce a notion of derivative for maps from $\prob(K)$ to $\R$. Among the many possible notions available in literature, we follow \cite{cardaliaguetcirantporretta}. We start from the notion of linear functional derivative (see \cite{carmonadelarue1}) and then we introduce the definition of $L$-derivative (which turns out to coincide with the one introduced by P.-L. Lions in \cite{lions}, see for instance \cite{carmonadelarue1}).
\begin{definition}[Linear functional derivative]\label{def: lfder}
	We say that a function $u\colon\prob(K)\to\R$ has linear functional derivative if there exists a map 
	\begin{equation*}
		\lf u\colon\prob(K)\times K\to\R,	
	\end{equation*}
	which is jointly continuous and such that for every $\mu,\nu\in\prob(K)$ it holds
	\begin{equation}
		u(\mu)-u(\nu) = \int_0^1\int_{K}\lf u(\theta \mu + (1-\theta)\nu,x)[\mu - \nu](\de x)\de \theta.
	\end{equation}
	We denote by $\mathrm{C^1}(\prob(K))$ the space of functions that are differentiable in linear functional sense.
\end{definition}
\begin{remark}
	On $\prob(K)$ the linear functional derivative is defined up to an additive constant. However, we can choose the versione such that, for every $\mu\in\prob(K)$,
\begin{equation*}
	\int_K\lf u (\mu,x)\mu(\de x) = 0.
\end{equation*}
\end{remark}
\begin{remark}
	Iterating Definition \ref{def: lfder} (keeping the extra spatial variable in $\lf u$ fixed), one can give a notion of second-order derivative. In particular $\lf^2 u$ will be a continuous mapping from $\prob(K)\times K\times K$ to $\R$. Thus we can introduce the spaces $\mathrm{C}^2(\prob(K))$. Analogously, this definition can be extended up to the $k$-th order, $k\in\N$.
\end{remark}
\begin{definition}[$L$-derivative]
	We say that a function $u$ is in $\mathrm{C^2_L}(\prob(K))$ if:
	\begin{enumerate}[i.]
		\item $u\in\mathrm{C^2}(\prob(K))$;
		\item for every $\mu\in\prob(K)$, the maps $ K\ni x\mapsto\lf u(\mu,x)$ and $ K\times K\ni(x,y)\mapsto\lf^2u(\mu,x,y)$ are twice continuously differentiable, with jointly continuous derivatives over $\prob(K)\times K$ and $\prob(K)\times K\times K$.
	\end{enumerate}
	We define the first and second-order $L$-derivatives as
	\begin{equation*}
		\diff_\mu u(\mu,x):=\diff_x \lf u(\mu,x)\in \R^d,\quad\diff_\mu^2 u(\mu,x,y):=\diff_x\diff_y^\top \lf^2 u(\mu,x,y)\in\R^{d\times d},
	\end{equation*}
	for every $\mu\in\prob(K)$ and $x,y\in K$.
\end{definition}
\begin{example}\label{ex: cyl}
	Let $n\in\N$, $g\colon\R^n\to\R$ be in $\mathrm{C^2_b}(\R^d)$ and let $\varphi_i\colon K\to\R$ be in $\mathrm{C^1}(K)$ for every $i=1,\dots,n$. We define $f\colon\prob(K)\to\R$ by
	\begin{equation*}
		f(\mu):=g(\scalprod{\mu}{\varphi_1},\dots,\scalprod{\mu}{\varphi_n}) = g(\scalprod{\mu}{\boldsymbol\varphi}).
	\end{equation*}
	Then, $f\in\mathrm{C^2_L}(\prob(K))$ and for every $\mu\in\prob(K)$ and $x,y\in K$ it holds
	\begin{align*}
		\lf f (\mu,x) &= \sum_{i=1}^n\partial_ig(\scalprod{\mu}{\boldsymbol\varphi})\varphi_i(x),&\lf^2 f (\mu,x,y) = \sum_{i,j=1}^n\partial^2_{ij}g(\scalprod{\mu}{\boldsymbol\varphi})\varphi_i(x)\varphi_j(y),\\
		\diff_\mu f (\mu,x) &= \sum_{i=1}^n\partial_ig(\scalprod{\mu}{\boldsymbol\varphi})\diff_x\varphi_i(x),&\diff_\mu^2 f (\mu,x,y) = \sum_{i,j=1}^n\partial^2_{ij}g(\scalprod{\mu}{\boldsymbol\varphi})\diff_x\varphi_i(x)\diff_y^\top\varphi_j(y).
	\end{align*}
\end{example}
\subsection{Stochastic filtering framework}\label{subsec: filt}
Let $(\Omega,\mathcal{F},\mathbb{P})$ be a probability space endowed with a filtration $\{\mathcal{F}_t\}_{t\in[0,T]}$ that satisfies the usual conditions and let $I = \{I_t,t\in[0,T]\}$ be a $d$-dimensional $\{\mathcal{F}_t\}$-Brownian motion. Let  $b,h\colon\R^d\to\R^d,\sigma,\bar\sigma\colon\R^d\to\R^{d\times d}$ and let us define two differential operators $A\colon\mathrm{C^2}(K)\to\mathrm{C}(K)$ and $B\colon\mathrm{C^2}(K)\to\mathrm{C}(K;\R^d)$ by
\begin{align}\label{eqn: operators}
	A\varphi := b^\top\diff_x\varphi + \frac{1}{2}\tr\left\{\sigma\sigma^\top\diff^2_x\varphi\right\} + \frac{1}{2}\tr\left\{\bar\sigma\bar\sigma^\top\diff^2_x\varphi\right\},\quad B\varphi :=\bar\sigma^\top(\diff_x\varphi).
\end{align}
We consider the Kushner-Stratonovich equation in weak form: 
\begin{equation}\label{eqn: ks}
	\de\scalprod{\pi_t}{\varphi} = \scalprod{\pi_t}{A\varphi}\de t +\left( \scalprod{\pi_t}{h\varphi + B\varphi}^\top-\scalprod{\pi_t}{\varphi}\scalprod{\pi_t}{h}^\top\right)\de I_t,\quad\varphi\in\mathrm{C^2}(K).
\end{equation}
\begin{remark}
This equation can be related to a stochastic filtering problem with signal $X=\{X_t,t\in[0,T]\}$ and observation $Y = \{Y_t,t\in[0,T]\}$ such that:
\begin{equation}\label{eqn: filt}
\begin{aligned}
	\de X_t &= b(X_t)\de t +\sigma(X_t)\de B_t + \bar\sigma(X_t)\de W_t,\\
	\de Y_t &= h(X_t)\de t + W_t,\\
\end{aligned}
\end{equation}
where $W = \{W_t,t\in[0,T]\}$ and $B = \{B_t,t\in[0,T]\}$ are two independent and $d$-dimensional Brownian motions. If we take $I$ as the innovation process associate to \eqref{eqn: filt} (see for instance \cite[Chapter 3]{baincrisan} for more details), then the equation \eqref{eqn: ks} is the one solved by the optimal filter $\hat\pi_t = \mathcal{L}(X_t\vert\mathcal{F}^Y_t)$, where $\mathcal{F}^Y_t:=\sigma\left(Y_s,0\leq s\leq t\right)$. However, we will look at \eqref{eqn: ks} in the spirit of \cite{szpirglas}, without relying directly on the filtering problem.
\end{remark}

Following \cite{heunislucic,szpirglas}, we say that the pair $\{(\Omega,\mathcal{F}, \{\mathcal{F}_t\},\mathbb{P}),(\pi,I)\}$ is a weak solution to the equation \eqref{eqn: ks} if
	\begin{enumerate}[i.]
		\item $(\Omega,\mathcal{F}, \{\mathcal{F}_t\},\mathbb{P})$ is a complete filtered probability space;
		\item $I$ is a $d$-dimensional $\{\mathcal{F}_t\}$-Brownian motion;
		\item $\pi= \{\pi_t,t\in[0,T]\}$ is a probability measure-valued continuous $\{\mathcal{F_t}\}$-adapted process such that
			\begin{equation*}
				\P{\int_0^T\sum_{i=1}^d\pi_s(\lvert h_i\rvert)^2\de s<+\infty} = 1,
			\end{equation*}
			and equation \eqref{eqn: ks} holds for every $t\in[0,T]$, almost surely.
	\end{enumerate}
Now we give some conditions that guarantee the existence, uniqueness (in law) and Markov property of a weak solution to the filtering equation \eqref{eqn: ks}. A complete discussion can be found in \cite{szpirglas,heunislucic}.
\begin{hypothesis}\label{hp: general}The following conditions hold:
	\begin{enumerate}[a.]
		\item $b,\sigma,\bar\sigma,h$ are Borel-measurable and bounded;
		\item $b,\sigma,\bar\sigma$ are Lipschitz continuous;
		\item for every $x\in K$ the matrix $\sigma(x)\sigma(x)^\top$ is a positive definite.
	\end{enumerate}
\end{hypothesis}
Moreover, we will ask for the following invariance property: 
\begin{hypothesis} \label{hp: invariance}
The probability measures-valued process $\pi$ has trajectories confined in the subset $\prob(K)$, $K\subset\R^d$ compact.
\end{hypothesis}
For instance, this is the case when $\pi$ is the filter associated to a signal process $X=\{X_t,t\in[0,T]\}$ with trajectories confined in a compact subset $K\subset\R^d$. Some conditions on $b,\sigma,\bar\sigma$ that guarantee this property are stated in Remark \ref{rmk: invarianceconditions}.

\begin{remark}\label{rmk: invarianceconditions}
	Let $G$ be a subset of $\R^d$. The invariance of a stochastic process with respect to $G$ has been intensively studied and many criteria that guarantee this property are available in the literature. For instance, following \cite[Section 12.2]{friedman2}, let $G\subset\R^d$ be closed and with $\mathrm{C^3}$ connected boundary $\partial G$. Let $\nu$ be the outward normal to $\partial G$ and let us define $\rho(x):=\inf_{y\in G}\lvert x-y\rvert$ the distance from $x\in\R^d$ to $G$. If the conditions 
	\begin{align*}
		\sum_{i,j=1}^da_{ij}\nu_i\nu_j=0,\qquad
		\sum_{i=1}^db_i\nu_i + \frac{1}{2}\sum_{i,j=1}^da_{ij}\partial^2_{ij}\rho\geq 0\quad\text{on}\ \partial G,
	\end{align*}
	hold, where $a = \sigma\sigma^\top + \bar\sigma\bar\sigma^\top$, then $G$ is invariant for $X$. Other less restrictive conditions can be found in \cite{aubindaprato, dapratofrankowska04,dapratofrankowska07} and in the references therein. 
\end{remark}
\begin{remark}
	All the following results do not change when $K$ is a compact manifold as long as we consider the right differential operators. An example is when the signal process lives on the $d$-dimensional torus $\mathbb{T}^d$.
\end{remark}
\section{Viscosity solutions to the backward Kolmogorov equations}\label{sec: viscosity}
The aim of this section is to continue the study of the backward Kolmogorov equation associated to the Kushner-Stratonovich equation \eqref{eqn: ks}. This partial differential equation of parabolic type over $\prob(K)\times[0,T]$ has been introduced in \cite{martini}. Here we present a well posedness result for viscosity solutions, which is different from the one available in \cite{martini}, where existence and uniqueness of classical solutions have been proved under more restrictive requirements on the final condition. 

Before writing down the equation, we introduce the operator $\mathcal{L}\colon\mathrm{C^2_L}(\prob(K))\to\mathrm{C}(\prob(K))$, which is the infinitesimal generator associated to the Kushner-Stratonovich equation. In particular, it is defined by
\begin{multline*}
	(\mathcal{L}u)(\mu) = \scalprod{\mu}{A\lf u(\mu)} + \frac{1}{2}\scalprod{\mu\otimes\mu}{(h+B-\scalprod{\mu}{h})^\top(h+B-\scalprod{\mu}{h})\lf^2u(\mu)}\\
	= \scalprod{\mu}{b^\top\diff_\mu u (\mu)} + \frac{1}{2}\scalprod{\mu}{\tr\{\diff_x\diff_\mu u(\mu)\sigma\sigma^\top\}}+  \frac{1}{2}\scalprod{\mu}{\tr\{\diff_x\diff_\mu u(\mu)\bar\sigma\bar\sigma^\top\}}\\
	\quad+\frac{1}{2}\scalprod{\mu\otimes\mu}{\lf^2 u(\mu)h^\top h} + \frac{1}{2}\scalprod{\mu\otimes\mu}{\tr\{\diff^2_\mu u(\mu)\bar\sigma\bar\sigma^\top\}}+\frac{1}{2}\scalprod{\mu\otimes\mu}{\lf^2u(\mu)}\lvert\scalprod{\mu}{h}\rvert^2\\
	\quad+\scalprod{\mu\otimes\mu}{h\cdot\bar\sigma^\top\lf\diff_\mu u(\mu)}-\scalprod{\mu\otimes\mu}{\lf^2u(\pi)h}^\top\scalprod{\mu}{h}-\scalprod{\mu\otimes\mu}{\bar\sigma^\top\lf\diff_\mu u(\mu)}^\top\scalprod{\mu}{h},
\end{multline*}
where $A,B$ are the operators defined in \eqref{eqn: operators} and $b,\sigma,\bar\sigma,h$ are the mappings introduced in Section \ref{subsec: filt}. For the second order terms we used the notation 
\begin{equation*}
	\scalprod{\mu\otimes\mu}{(f_1^\top f_2) g} :=\int_{K\times K}f_1(x)^\top f_2(y)g(x,y)\mu(\de x)\mu(\de y), 
\end{equation*}
$f_1,f_2\colon K\to\R$ and $g\colon K\times K\to\R$.\\

We can now introduce the backward Kolmogorov equation associated to \eqref{eqn: ks} with final condition $\Phi\in\mathrm{C}(\prob(K))$:
\begin{equation}\label{eqn: kolmogorov}
	\begin{cases}
		-\partial_t u(\mu,t) - \mathcal{L} u (\mu,t) = 0,\quad &(\mu,t)\in\prob(K)\times[0,T),\\
		u(\mu,T) = \Phi(\mu),\quad &\mu\in\prob(K).
	\end{cases}
\end{equation}
A first notion of solution is the classical one, in which we ask the solution to be in $\mathrm{C^2_{L}}(\prob(K))$ with respect to the measure argument and in $\mathrm{C^1}([0,T])$ with respect to the time. We will denote this class of functions by $\mathrm{C^{2,1}_L}(\prob(K)\times[0,T])$.
\begin{definition}
	We say that $u\colon\prob(K)\times[0,T]\to\R$ is a classical solution to \eqref{eqn: kolmogorov} if it is in $\mathrm{C^{2,1}_L}(\prob(K)\times[0,T])$, with all the bounds on the derivatives that are uniform, and if it satisfies the backward equation \eqref{eqn: kolmogorov}.
\end{definition}
Assuming more regularity on the final condition $\Phi$, it has been proved in \cite{martini} that there exists a unique classical solution to \eqref{eqn: kolmogorov}.
\begin{theorem}\label{thm: exuniqclassical}
	Let $\{(\Omega,\mathcal{F}, \{\mathcal{F}_t\},\mathbb{P}),(\pi^{\mu,s},I)\}$ be a weak solution to \eqref{eqn: ks}, starting at time $s\in[0,T]$ from $\mu\in\prob(K)$ and that satisfies the invariance Hypothesis \ref{hp: invariance}. Moreover, let $\Phi\in\mathrm{C^2_L}(\prob(K))$ and let Hypothesis \ref{hp: general} holds. Then, there exists a unique classical solution to \eqref{eqn: kolmogorov} given by 
	\begin{equation}\label{eqn: reprformula}
		u(\mu,t) = \E{\Phi(\pi^{\mu,t}_T)},\quad (\mu, t)\in\prob(K)\times[0,T].
	\end{equation}
\end{theorem}
\begin{remark}
	In \cite{martini} the process $\pi^{\mu,s}$ is built starting from another masure-valued process, namely a weak solution to the Zakai equation (see for instance \cite{heunislucic, szpirglas} for more details). However, here we avoid to emphasize this to keep a simpler notation.
\end{remark}
We want to investigate the well-posedness of \eqref{eqn: kolmogorov} when the final condition has lower regularity, namely $\Phi\in\mathrm{C}(\prob(K))$. In this case classical solutions may not exist, and a natural way to face the problem is to consider a generalized notion of solution. In particular we will focus on the notion of viscosity solution:
\begin{definition} A  upper semicontinuous (resp. lower semicontinuous) function $u\colon\prob(K)\times [0,T]\to\R$ is a viscosity subsolution (resp. supersolution) to equation \eqref{eqn: kolmogorov} if:
\begin{enumerate}[i.]
	\item $u(\mu,T)\leq$ (resp. $\geq$) $\Phi(\mu)$, for every $\mu\in\prob(K)$;
	\item for every $(\mu,t)\in \prob(K)\times[0,T)$ and any $\varphi\in \mathrm{C^{2,1}_L}(\prob(K)\times[0,T])$ such that $u-\varphi$ has a maximum (resp. minimum) at $(\mu,t)$ with value 0, then \eqref{eqn: kolmogorov} holds for $\varphi$ with the inequality sign $\leq$ (resp. $\geq$) instead of the equality.
\end{enumerate}
We say that $u$ is a viscosity solution if it is both a viscosity subsolution and a viscosity supersolution.
\end{definition}
To conclude this preliminary part, we state an It\^o formula for the composition of a $\mathrm{C^2_L}(\prob(K))$ function with a weak solution $\pi$ of the equation \eqref{eqn: ks} (\cite{martini}). 
\begin{proposition}\label{prop: ito}
Let $\{(\Omega,\mathcal{F}, \{\mathcal{F}_t\},\mathbb{P}),(\pi^{\mu,s},I)\}$ be a weak solution to \eqref{eqn: kolmogorov} that satisfies the invariance Hypothesis \ref{hp: invariance} and let Hypothesis \ref{hp: general} hold. Moreover, let $u\in\mathrm{C^2_L}(\prob(K))$. Then, for evrey $t\in[s,T]$
\begin{multline*}
	u(\pi^{\mu,s}_t) = u(\mu) + \int_s^t\scalprod{\pi^{\mu,s}_r}{A\lf u(\pi^{\mu,s}_r)} \de r + \int_s^t \scalprod{\pi^{\mu,s}_r}{(h + B - \scalprod{\pi^{\mu,s}_r}{h})^\top\lf u(\pi^{\mu,s}_r)}\de I_r\\
	+\frac{1}{2}\int_s^t\scalprod{\pi^{\mu,s}_r\otimes\pi^{\mu,s}_r}{(h + B - \scalprod{\pi^{\mu,s}_r}{h})^\top(h + B - \scalprod{\pi^{\mu,s}_r}{h})\lf^2 u(\pi^{\mu,s}_r)}\de r,
\end{multline*}
almost surely.
\end{proposition}
\begin{remark}
	The results in Theorem \ref{thm: exuniqclassical} and in Proposition \ref{prop: ito} hold more generally when $\pi$ takes values in $\ptwo(\R^d)$, the space of probability measures over $\R^d$ with finite second moment.
\end{remark}
\subsection{Existence of viscosity solutions}\label{ssec: exist}
Let us introduce a function $u$ by the usual probabilistic representation formula \eqref{eqn: reprformula}, where $\Phi\in\mathrm{C}(\prob(K))$ is the final condition of \eqref{eqn: kolmogorov} and $\pi^{\mu,s}$ is a weak solution to the Kushner-Stratonovich equation starting at $s\in[0,T]$ from $\mu\in\prob(K)$. 
\begin{proposition}\label{prop: existence}
	Let $\Phi\in\mathrm{C}(\prob(K))$, let Hypothesis \ref{hp: general} holds and let $\pi^{\mu,s}$ a solution to \eqref{eqn: ks} satisfying the invariance Hypothesis \eqref{hp: invariance}. Then the function 
	\begin{equation}\label{eqn: reprformula}
		u(\mu,t) := \E{\Phi(\pi^{\mu,t}_T)},\quad (\mu, t)\in\prob(K)\times[0,T],
	\end{equation}
	is a viscosity solution to the backward Kolmogorov equation \eqref{eqn: kolmogorov}.
\end{proposition}
\begin{proof}
	We only prove that $u$ is a subsolution, since proving that it is a supersolution is analogue. 
	First, thanks to the continuity of $\Phi$ and $\pi^{\mu,s}$, we notice that $u$ is continuous and bounded.  Moreover $ u(\mu,T) = \E{\Phi(\pi^{\mu,T}_T)} = \Phi(\mu)$ for every $\mu\in\prob(K)$.
	
	Now, let us fix $(\mu,t)\in\prob(K)\times[0,T)$ and let us consider $\varphi\in \mathrm{C^{2,1}_L}(\prob(K)\times[0,T])$ such that $\varphi(\mu,t) = u(\mu,t)$ and $\varphi(\nu,s)\geq u(\nu,s)$ for every $(\nu,s)\in\prob(K)\times[0,T)$. Then, for $h>0$ it holds that
	\begin{equation*}
		0\leq \frac{\varphi(\pi^{\mu,t}_{t+h},t+h) - u(\pi^{\mu,t}_{t+h},t+h) }{h},
	\end{equation*}
	and, by taking the expectation combined with It\^o formula and Markov property of $\pi^{\mu,t}$, we obtain 
	\begin{equation*}
		0\leq \frac{1}{h}\E{\int_t^{t+h}\mathcal{L}\varphi(\pi_\tau^{\mu,t},t+h)\de\tau} + \frac{\varphi(\mu,t+h)-\varphi(\mu,t)}{h}.
	\end{equation*}
	Finally, thanks to the regularity of $\varphi$ and its derivatives, we can pass to the limit as $h\to 0$ and obtain
	\begin{equation*}
		0\leq \mathcal{L}\varphi(\mu,t) + \partial_t \varphi (\mu,t).
	\end{equation*}
\end{proof}
\begin{remark}
	This existence result is still valid even if we drop Hypothesis \ref{hp: invariance} and we look at the equation over $\ptwo(\R^d)$.
\end{remark}
\subsection{Comparison theorem and uniqueness}\label{ssec: uniq}
Here we state and prove a comparison principle for viscosity solutions to \eqref{eqn: kolmogorov}, which is a crucial tool to prove the uniqueness. The technique for the proof is inspired by the one used for Theorem II.1 in \cite{lions83_2}, and by its recent refinement used in \cite{cossogozzikharroubiphamrosestolato2}. In particular, we use the solution $u$ introduced by the representation formula \eqref{eqn: reprformula} as a reference to obtain partial comparison results. 
\begin{theorem}\label{thm: comparison}
Let $u_1,u_2\colon \prob(K)\times[0,T]\to\R$ be respectively a viscosity subsolution and supersolution of \eqref{eqn: kolmogorov}.
Moreover, let Hypothesis \ref{hp: general} and Hypothesis \ref{hp: invariance} hold. Then, $u_1(\mu,t)\leq u_2(\mu,t)$ for every $(\mu,t)\in \prob(K)\times[0,T]$.
\end{theorem}
\begin{proof}
We want to prove that $u_1\leq u \leq u_2$, where $u$ is given by \eqref{eqn: reprformula}. The proof is divided in five steps and we show only that $u_1 \leq u$. Indeed we can prove $u \leq u_2$ by noticing that $u_2$ is a subsolution of \eqref{eqn: kolmogorov} with final condition $-\Phi$ and by taking $-u$ instead of $u$.\\\\
By contradiction, let us assume that there exists $(\mu_0,t_0)\in \prob(K)\times[0,T]$ such that
\begin{equation}\label{eqn: hpcontr}
	u_1(\mu_0,t_0) - u(\mu_0,t_0)>0.
\end{equation}
First, we notice that $t_0<T$, indeed $u_1(\mu, T)\leq \Phi(\mu) = u(\mu,T)$ for every $\mu\in \prob(K)$, thanks to the subsolution property of $u_1$.\\\\
\textit{Step 1. }Let $\{\Phi_n\}_{n\in\N}\subset\mathrm{C^2_L}( \prob(K))$ be the sequence that converges uniformly to $\Phi\in\mathrm{C}( \prob(K))$ given by Proposition \ref{prop: sw}. For every $n\in\N$, thanks to Theorem \ref{thm: exuniqclassical} we can set
\begin{equation*}
	u_n(\mu,t) := \E{\Phi_n(\pi^{\mu,t}_T)},\quad (\mu,t)\in \prob(K)\times[0,T],
\end{equation*}
which is the unique classical solution to equation \eqref{eqn: kolmogorov} with final condition $\Phi_n$. Since $\{\Phi_n\}_{n\in\N}$ converges uniformly to $\Phi$, the sequence is bounded by a positive constant $M=M(\Phi)$. Then, dominated convergence allows us to conclude that
\begin{equation}\label{eqn: pointconv}
	\lim_{n\to+\infty}u_n(\mu,t) =  u(\mu,t),\quad (\mu, t)\in  \prob(K)\times[0,T].
\end{equation}
\noindent\textit{Step 2. }Let us introduce
\begin{equation}\label{eqn: defubar}
	\bar{u}_1 (\mu,t) := \mathrm{e}^{t-t_0}u_1(\mu,t),\quad \bar{u}_n  (\mu,t) := \mathrm{e}^{t-t_0}u_n(\mu,t),\quad (\mu,t)\in \prob(K)\times[0,T),
\end{equation}
and
\begin{equation*}
	\bar{\Phi} (\mu):=\mathrm{e}^{T-t_0}\Phi(\mu),\quad\bar{\Phi}_n  (\mu):=\mathrm{e}^{T-t_0}\Phi_n(\mu),\quad \mu\in \prob(K).
\end{equation*}
Then, it is easy to see that $\bar{u}_n $ is a classical solution of the equation
\begin{equation}\label{eq: bknl}
	\begin{cases}
		-\partial_t v(\mu,t) - \mathcal{L} v (\mu,t) = - v(\mu,t),\quad &(\mu,t)\in \prob(K)\times[0,T),\\
		v(\mu,T) = \bar{\Phi} _n(\mu),\quad &\mu\in \prob(K),
	\end{cases}
\end{equation}
and, analogously, that $\bar{u}_1 $ is a viscosity subsolution of  
\begin{equation*}
	\begin{cases}
		-\partial_t v(\mu,t) - \mathcal{L} v (\mu,t) = - v(\mu,t),\quad &(\mu,t)\in \prob(K)\times[0,T),\\
		v(\mu,T) = \bar{\Phi} (\mu),\quad &\mu\in \prob(K).
	\end{cases}
\end{equation*}
\textit{Step 3. } For every $\lambda > 0$, let us define the mapping $\Psi_\lambda\colon\prob(K)\times[0,T]\to\R$ as
\begin{equation}\label{eqn: defPsi}
	\Psi_\lambda(\mu,t):=\bar{u_1}(\mu,t) - \bar{u}_n(\mu,t) - \lambda(t-t_0)^2 - \lambda \sfd_2(\mu_0,\mu)^2.
\end{equation}
Since $\bar{u}_1$ is upper semicontinuous over a compact space, and $\bar{u}_n $ is continuous and bounded uniformly in $n$, we have that $(\mu_\lambda,t_\lambda)\in\prob(K)\times[0,T]$ is a global maximum for $\Psi_\lambda$. Moreover, from \eqref{eqn: defPsi} it follows that
\begin{equation}\label{eqn: estpsi}
	\bar{u}_1(\mu_0,t_0) - \bar{u}_n(\mu_0,t_0) = \Psi_\lambda(\mu_0,t_0)\leq\Psi(\mu_\lambda,t_\lambda)\leq\bar{u}_1(\mu_\lambda,t_\lambda) - \bar{u}_n(\mu_\lambda,t_\lambda),
\end{equation}
and so 
\begin{equation}\label{eqn: estdistance}
	\lambda(t_\lambda-t_0)^2 + \lambda \sfd_2(\mu_0,\mu_\lambda)^2\leq \left(\bar{u}_1(\mu_\lambda,t_\lambda) - \bar{u}_n(\mu_\lambda,t_\lambda)\right) - \left(\bar{u}_1(\mu_0,t_0) - \bar{u}_n(\mu_0,t_0)\right).
\end{equation}
Thanks to \eqref{eqn: hpcontr} and \eqref{eqn: pointconv}, for $n$ large enough we have that $\left(\bar{u}_1(\mu_0,t_0) - \bar{u}_n(\mu_0,t_0)\right)>0$. Combining this with the fact that $\bar{u}_1$ is bounded from above and $\bar{u}_n $ is bounded uniformly in $n$, from \eqref{eqn: estdistance} we obtain that there exists a constant $M>0$ such that
\begin{equation}\label{eqn: estmdelta}
	\lvert t_\lambda-t_0\rvert\leq\sqrt{\frac{M}{\lambda}},\quad \sfd_2(\mu_0,\mu_\lambda)\leq\sqrt{\frac{M}{\lambda}}.
\end{equation}
%
\textit{Step 4. }Now we show that $ t_\lambda<T$. By contradiction, let $t_\lambda = T$. From \eqref{eqn: estpsi}, 
\begin{equation*}
	u_1(\mu_0,t_0) - u_n(\mu_0,t_0) = \bar{u}_1(\mu_0,t_0) - \bar{u}_n(\mu_0,t_0) \leq \bar{u}_1(\mu_\lambda,T) - \bar{u}_n(\mu_\lambda,T),
\end{equation*}
and then
\begin{equation}
	u_1(\mu_0,t_0) - u_n(\mu_0,t_0)\leq \mathrm{e}^{T - t_0}\left(\Phi(\mu_\lambda ) - \Phi_n(\mu_\lambda)\right).
\end{equation} 
Letting $n\to\infty$, since $\Phi_n\to\Phi$ uniformly (see Proposition \ref{prop: sw}) and $u(\mu_0,t_0)\to u_n(\mu_0,t_0)$, we obtain
\begin{equation*}
	u_1(\mu_0,t_0) - u(\mu_0,t_0)\leq0,
\end{equation*}
which contradicts \eqref{eqn: hpcontr}.\\\\
\textit{Step 5. }Let us define $\prob(K)\times[0,T]\ni(\mu,t)\mapsto\varphi_\lambda(\mu,t):=\lambda(t-t_0)^2 + \lambda \sfd_2(\mu_0,\mu)^2$ and let us notice that $\varphi_\lambda\in\mathrm{C^{2,1}_L}(\prob(K)\times[0,T])$ thanks to Lemma \ref{lemma: distancediffer}. We want to use $\bar{u}_n + \varphi_\lambda$ as test function for the viscosity subsolution property of $\bar{u}_1$. First, $\bar{u}_n + \varphi_\lambda$ is in $\mathrm{C^{2,1}_L}(\prob(K)\times[0,T])$ since $\bar{u}_n,\varphi_\lambda\in\mathrm{C^{2,1}_L}(\prob(K)\times[0,T])$. Moreover, from \textit{Step 3} we have that $\bar{u}_1 - (\bar{u}_n + \varphi_\lambda)$ has a maximum in $(\mu_\lambda,t_\lambda)$ and from \textit{Step 4} it holds that $t_\lambda<T$. 

Thus, we can actually use $\bar{u}_n +\varphi_\lambda$ as a test function in $(\mu_\lambda,t_\lambda)$ and, by the viscosity subsolution property of $\bar{u}_1 $, obtain
\begin{equation*}
	\bar{u}_1(\mu_\lambda,t_\lambda)\leq\partial_t(\bar{u}_n+\varphi_\lambda)(\mu_\lambda,t_\lambda) + \mathcal{L}(\bar{u}_n+\varphi_\lambda)(\mu_\lambda,t_\lambda).
\end{equation*}
Since $\bar{u}_n$ is a classical solution of \eqref{eq: bknl}, this entails
\begin{equation*}
	\bar{u}_1(\mu_\lambda,t_\lambda) - \bar{u}_n(\mu_\lambda,t_\lambda)\leq \partial_t\varphi_\lambda(\mu_\lambda,t_\lambda) + \mathcal{L}\varphi_\lambda(\mu_\lambda,t_\lambda).
\end{equation*}
Now we can use the estimates \eqref{eqn: estmdelta} and the bound in Remark \ref{rmk: lbounded} to obtain, for a positive constant $C>0$,
\begin{equation*}
	\bar{u}_1(\mu_\lambda,t_\lambda) - \bar{u}_n(\mu_\lambda,t_\lambda)\leq 2\lambda(t_\lambda-t_0) + \lambda(\mathcal{L}\rho_{\mu_0})(\mu_\lambda)\leq 2\sqrt{\lambda M} + \lambda C,
\end{equation*}
where $\rho_{\mu_0}$ is defined as in \eqref{eqn: defrho}.
Finally, by \eqref{eqn: estpsi} and \eqref{eqn: defubar}, it follows that
\begin{equation}
	u_1(\mu_0,t_0) - u_n(\mu_0,t_0) = \bar{u}_1(\mu_0,t_0) - \bar{u}_n(\mu_0,t_0)\leq 2\sqrt{\lambda M} + \lambda C(\mu_0),
\end{equation}
which contradicts \eqref{eqn: hpcontr} if we let $\lambda\to0$ and $n\to\infty$. 
\end{proof}
\begin{corollary}\label{cor: uniq}
	Let $\Phi\in\mathrm{C}(\prob(K))$ and let Hypothesis \ref{hp: general} and Hypothesis \ref{hp: invariance} hold. Then
	\begin{equation*}
		u(\mu,t) = \E{\Phi(\pi^{\mu,t}_T)},\quad (\mu, t)\in\prob(K)\times[0,T],
	\end{equation*}	
	is the unique viscosity solution to the backward Kolmogorov equation \eqref{eqn: kolmogorov}.
\end{corollary}
\begin{proof}
	Let $u$ be the solution to \eqref{eqn: kolmogorov} defined by \eqref{eqn: reprformula} and let $v$ be another viscosity solution. Since in particular $u$ is a subsolution and $v$ is a supersolution, Theorem \ref{thm: comparison} tells us that $u\leq v$. Changing the role of $u$ and $v$ we obtain that $v\leq u$ and so the two solutions coincide.
\end{proof}

\appendix
\section{Some analytical results}\label{app: analytical}
Here we collect some analytical tools necessary in the proof of the comparison principle (Theorem \ref{thm: comparison}). In particular we need to study the differentiability of the distance $\sfd_2$ introduced in \eqref{eqn: distwang}. Moreover, we also need to introduce a way to approximate continuous functions over $\prob(K))$ with functions in $\mathrm{C^2_{L}}(\prob(K))$.\\\\
For a family $\{f_k\}_{k\in\N}$ dense in $\mathrm{C^\infty_c}(K)$, with $\norm{f_k}_\infty\leq 1$ and containing the constant function equal to one, and for a sequence $\{q_k\}_{k\in\N}\subset[1,+\infty)$, we recall that
\begin{equation*}
	\sfd^2_2(\mu,\nu) = \sum_{k=1}^{+\infty}\frac{1}{2^kq_k}\scalprod{\mu-\nu}{f_k}^2,\quad \mu,\nu\in\prob(K). 
\end{equation*} For what follows, we need to choose
\begin{equation*}
	q_k:=\max\left\{1,\norm{\diff_x f_k}_\infty,\norm{\diff^2_x f_k}_{\infty},\norm{\diff_x f_k}_\infty^2\right\},\quad k\in\N.
\end{equation*}
Let $\mu_0\in\prob(K)$ be fixed and let $\rho_{\mu_0}\colon\prob(K)\to[0,+\infty)$ be defined as
\begin{equation}\label{eqn: defrho}
	\rho_{\mu_0}(\mu):=\sfd_2(\mu_0,\mu)^2,\quad\mu\in\prob(K).
\end{equation}

\begin{lemma}\label{lemma: distancediffer}
Let $\mu_0\in\prob(K)$ be fixed. Then, the mapping $\rho_{\mu_0}$ is in $\mathrm{C^2_{L}}(\prob(K))$ with derivatives uniformly bounded with respect to $\mu_0$.
\end{lemma}
\begin{proof}
	First, let us fix $k\in\N$ and show that the mapping $\mu\mapsto\scalprod{\mu - \mu_0}{f_k}^2=:g_k(\mu)$, where $f_k$ is in the family appearing in  \eqref{eqn: distwang}, is in $\mathrm{C^2_{L}}(\prob(K))$. By standard computations we get
	\begin{equation}
		\begin{aligned}
		\lf g_k (\mu,x) &= 2\scalprod{\mu - \mu_0}{f_k}f_k(x),&\quad &\mu\in\prob(K),x\in K,\\
		\lf^2 g_k (\mu,x,y) &= 2f_k(x)f_k(y),&\quad &\mu\in\prob(K),x,y\in K,
		\end{aligned}
	\end{equation}
	with $\norm{\lf g_k}_\infty \leq 4$ and $\norm{\lf^2g_k}_\infty\leq 2$ for every $k\in\N$. Moreover, we have that
	\begin{equation*}
		\begin{aligned}
		\diff_\mu g_k (\mu,x) &= 2\scalprod{\mu - \mu_0}{f_k}\diff_x f_k(x),&\quad &\mu\in\ptwo(K),x\in K,\\
		\diff_x\diff_\mu g_k (\mu,x) &= 2\scalprod{\mu - \mu_0}{f_k} \diff^2_xf_k(x),&\quad &\mu\in\ptwo(K),x\in K,\\
		\diff_\mu^2 g_k (\mu,x,y) &= 2\diff_x f_k(x)\diff_y^\top f_k(y),&\quad &\mu\in\ptwo(K),x,y\in K,\\
		\lf\diff_\mu g_k (\mu,x,y) &= 2\diff_x f_k(x) f_k(y),&\quad &\mu\in\ptwo(K),x,y\in K,
		\end{aligned}
	\end{equation*}
	with
	\begin{equation*} 
	\begin{aligned}
		\norm{\diff_\mu g_k}_\infty &\leq 4\norm{\diff_x f_k}_\infty,&\quad \norm{\diff_x\diff_\mu g_k }_\infty \leq& 4\norm{\diff^2_x f_k}_\infty,\\
		\norm{\diff_\mu^2g}_\infty&\leq 2\norm{\diff_x f_k}_\infty^2,&\quad\norm{\lf\diff_\mu g_k}_\infty\leq &2\norm{\diff_x f_k}_\infty. 
	\end{aligned}
	\end{equation*}
	To conclude, by dominated convergence we can bring the derivative inside the series and get
	\begin{equation*}
		\lf \rho_{\mu_0}(\mu,x) = \sum_{k=1}^{+\infty}\frac{1}{2^kq_k}\lf g_k(\mu,x),\quad \mu\in\prob(K),x\in K,
	\end{equation*}
	which is also jointly continuous. Moreover, recalling the definition of $\{q_k\}_{k\in\N}$, we have
	\begin{equation*}
		\norm{\lf \rho_{\mu_0}}_\infty\leq\sum_{k=1}^{+\infty}\frac{1}{2^kq_k}\norm{\lf g_k}_\infty\leq4.
	\end{equation*} We can proceed in the same way for the other derivatives and the bounds follow easily.
\end{proof}
\begin{remark}\label{rmk: lbounded}
	Let $\mathcal{L}$ be the differential operator introduced in Section \ref{sec: viscosity}. Under Hypothesis \ref{hp: general}, we have that from Lemma \ref{lemma: distancediffer} follows that there exists a constant (independent of $\mu_0$) $C >0$ such that  $\norm{\mathcal{L}\rho_{\mu_0}}_\infty= \sup_{\mu\in\prob(K)}\lvert(\mathcal{L}\rho_{\mu_0})(\mu)\rvert\leq C$. Indeed
	\begin{multline}
		\norm{\mathcal{L}\rho_{\mu_0}}_\infty \leq \norm{A\lf u}_\infty + \frac{1}{2}\norm{(h+B-\scalprod{\mu}{h})^\top(h+B-\scalprod{\mu}{h})\lf^2u}_\infty\\
		\leq 4\norm{b}_\infty +  2\norm{\sigma}^2_\infty + 3\norm{\bar{\sigma}}^2_\infty + 4\norm{h}^2_\infty + 4\norm{h}_\infty\norm{\bar{\sigma}}_\infty =:C.
	\end{multline}
\end{remark}
\quad\\
To conclude this auxiliary section, we provide a class of polynomials on the space of probability measures that can be used to approximate functions in $\mathrm{C}(\prob(K))$, which is the set of continuous real-valued functions over $\prob(K)$. This kind of approximation has been discussed for instance in \cite{cuchierolarssonsvalutoferro}.

\begin{definition}
	We define the set of polynomials over $\prob(K)$:
	\begin{equation}
		\mathscr{P}(\prob(K)):=\{\phi\colon\ptwo\to\R \text{ of the form } \phi(\mu) = f(\scalprod{\mu}{\psi_1},\dots,\scalprod{\mu}{\psi_n}),\text{ for some }n\in\N\},
	\end{equation}
	with $f\colon\R^n\to\R$ and $\psi_i\colon\R^d\to\R$ polynomials, $i=1,\dots,n$.
\end{definition}
\begin{remark} It holds that $\mathscr{P}(\prob(K))\subset\mathrm{C^2_{L}}(\prob(K))$. Indeed, every $\phi\in\mathscr{P}(\prob(K))$ is of the form discussed in Example \ref{ex: cyl}.
\end{remark}
One can prove (see for instance \cite[Section 2.3]{cuchierolarssonsvalutoferro}) that the family $\mathscr{P}(\prob(K))$ is an algebra in $\prob(K)$ that separates the points and which contains the constants. Thus, by Stone-Weiestrass theorem one can conclude that it is dense in $\mathrm{C}(\prob(K))$:
\begin{proposition}\label{prop: sw}
	The family $\mathscr{P}(\prob(K))$ is dense in $\mathrm{C_b}(\prob(K))$ with respect to the supremum norm over $\prob(K)$.
\end{proposition}
\subsection*{Acknowledgements} The author is grateful to Prof. Andrea Cosso for the helpful discussions and suggestions.
\printbibliography
\end{document}